\documentclass[12pt]{amsart}
\usepackage{xcolor}
\usepackage{graphicx}
\usepackage{dsfont}
\usepackage{float}

\newtheorem{theorem}{Theorem}
\newtheorem{lemma}{Lemma}

\newtheorem{remark}{Remark}

\newcommand{\be}{\begin{equation}}
\newcommand{\ee}{\end{equation}}
\newcommand{\bea}{\begin{eqnarray}}
\newcommand{\eea}{\end{eqnarray}}
\newcommand{\barr}{\begin{array}}
\newcommand{\earr}{\end{array}}

\newcommand{\bpar}{\begin{equation} \left\{ \begin{array}{lll}}
\newcommand{\epar}{ \end{array}\right. \end{equation} }
\newcommand{\eparn}{ \end{array} \right.}

\newcommand\leftmat{\left(\begin{array}{cc}}
\newcommand\rightmat{\end{array}\right)}
\newcommand\leftvec{\left(\begin{array}{c}}
\newcommand\rightvec{\end{array}\right)}

\title[The phenomenon of revivals]{The phenomenon of revivals on complex potential Schr{\"o}dinger's equation}
\author{
L. Boulton$^\beta$, G. Farmakis$^\varphi$,  B. Pelloni$^\pi$}

\address{
Heriot-Watt University \& Maxwell Institute  for the Mathematical Sciences, Edinburgh, Scotland.
}

\email{$^\beta$L.Boulton@hw.ac.uk, $^\varphi$G.Farmakis@hw.ac.uk, $^\pi$B.Pelloni@hw.ac.uk}

\date{3rd April 2024}

\begin{document}
\maketitle

\begin{abstract}
The mysterious phenomena of revivals in linear dispersive periodic equations was discovered first experimentally in optics in the 19th century, then rediscovered several times by theoretical and experimental investigations. While the term has been used systematically and consistently by many authors, there is no consensus on a rigorous definition.  In this paper, we describe revivals modulo a regularity condition in a large class of Schr{\"o}dinger's equations with complex bounded potentials. As we show, at rational times the solution is given explicitly by finite linear combinations of translations and dilations of the initial datum, plus an additional continuous term.
\end{abstract}

\

\medskip

\medskip

\newpage

\section{Introduction}

Recently there have been significant developments in the study of {revivals} in dispersive evolution equations \cite{smith2020revival}.  These phenomena, which are also called dispersive quantisations or Talbot effects, describe a surprising dichotomy in the pointwise behaviour of the solution of time-evolution equations at specific values of the time variable, the so-called {rational times},  compared to all other generic times. At rational times, the solution revives the shape of the initial datum by finite superpositions, reflections and re-scalings, with a prescribed simple combinatorial rule. See \cite{erdougan2016dispersive} and references therein. 

The majority of past investigations about revivals involve equations and boundary conditions with the property that the eigenpairs of the spatial operator can be found explicitly and satisfy precise conditions of modularity and periodicity. The prime example of this is the case of  linear dispersive equations with constant coefficients and periodic boundary conditions. In such cases, the techniques for detecting the times at which the revivals appear exploit the specific periodic matching of the eigenvalues, eigenfunctions and boundary conditions. With the help of summations of Gauss type, the infinite series representation of the solution then reduces to a finite sum, characterising the revivals explicitly. Naturally, the direct applicability of this approach is limited.

The purpose of this paper is to take a different point of view, by formulating the revivals phenomena in terms of perturbation theory. 
Concretely, we ask the question of whether a (large) class of equations exhibit them, modulo a regular perturbation. Earlier version of this concept can be found in the works \cite{rodnianski1999continued,cho2021talbot, erdogan2013talbotpaper,erdougan2013global}, about which we give details below.

\smallskip

We consider the class of linear Schr\"odinger equations 
\begin{align}\label{pbm1}
  & \partial_t u(x,t)=-i(-\partial_x^2+V(x))u(x,t), &&x\in(0,\pi), \;t>0,\nonumber \\
&u(0,t)=u(\pi,t)=0, && t>0, \\
   & u(x,0)=f(x), && x\in(0,\pi),\nonumber
\end{align}
with a complex-valued potential $V$, subject to Dirichlet boundary conditions, given an initial wavefunction $f\in L^2(0,\pi)$. Our goal is to detect revivals by perturbation from the case $V=0$. As we shall see below, in the large wavenumber asymptotic regime and for small enough $V$, the simple but non-trivial structure of \eqref{pbm1} supports the combinatorial argument, involving the Gauss summations, that is valid for the free-space equation.

Our contribution is summarised in the next theorem. It shows that the solution of \eqref{pbm1} at rational times support revivals modulo a continuous term. This result matches a similar earlier finding, reported in \cite{BFP}. Indeed, for $V=0$ and boundary conditions of the type $bu(0,t) = (1-b) \partial_{x} u(\pi,t)$ where $b\in (0,1)$ is a parameter, an analogous conclusion holds true.

\smallskip

Here and everywhere below, $f^{\mathrm{o}}$ denotes the odd, $2\pi$-periodic extension of the function $f$ and $\langle V \rangle=\frac{1}{\pi}\int_0^{\pi}V(x)\mathrm{d}x$ the mean of the potential function.

\begin{theorem} \label{theorem2}
Let $V\in H^2(0,\pi)$ with $\|V\|_{\infty}<\frac32$.
Then, for $p,q\in\mathbb{N}$ co-prime numbers, the solution $u(x,t)$ to \eqref{pbm1} at time $t=2\pi\frac{p}{q}$ is given by
\[
     u\Big(x,2\pi\frac{p}{q}\Big)=w\Big(x,2\pi\frac{p}{q}\Big)+\frac{1}{q} \ e^{-2\pi i\langle V\rangle \frac{p}{q}}
      \sum_{k,m=0}^{q-1} e^{ 2\pi i ( m\frac{k}{q}-m^2 \frac{p}{q} )} f^{\mathrm{o}}\Big(x-2\pi\frac{k}{q}\Big)
\] 
where $w(\cdot,t)\in \operatorname{C}(0,\pi)$ for all fixed $t>0$. 
\end{theorem}

In case $V$ is real-valued, the same conclusion holds with the weaker assumptions $V\in \operatorname{BV}(0,\pi)$ and $\|V\|_{\infty}<\infty$. In Section~\ref{optimality} we investigate whether the bound on $V$ is necessary in the complex case. 

We interpret the conclusion of Theorem~\ref{theorem2} by saying that \eqref{pbm1} supports a weak form of revival. Note that the result does not depend on the orthogonality of the family of eigenfunctions. Our choice of Dirichlet boundary conditions, corresponding to potential barriers at both ends of the finite segment, exhibits all the features of our methodology but is free from unnecessary complications. Indeed, while Neumann and other separated boundary conditions can be treated via a similar approach, in these cases multiplicities and the possibility of eigenvalues which are not semi-simple lead to additional technical distractions.

\medskip

Phenomena similar to the one described here had already been observed. The investigations conducted in \cite{rodnianski1999continued} and \cite{cho2021talbot} lead to versions of Theorem~\ref{theorem2} for periodic $V$ and periodic boundary conditions. The method of proof in these works is different from the one below. It relies on Duhamel's representation and an analysis  of the solution in periodic Besov spaces. In a separate development, for the periodic cubic non-linear Schr\"{o}dinger equation \cite{erdogan2013talbotpaper} and the Korteweg-de Vries equation \cite{erdougan2013global}, it was proved  that at all times the difference between the solution and the linear time-evolution is more regular than the initial datum. This directly implies the appearence of weak revivals at rational times also for these two non-linear equations. Numerical evidence of this effect in the non-linear setting was reported in \cite{chen2013dispersion, chen2014numerical}.    Our findings complement all these investigations.

The proof of Theorem~\ref{theorem2} that we present shows how to apply directly classical perturbation expansions in order to derive existence of revivals. This approach has two main implications worth mentioning. On the one hand, Theorem~\ref{theorem2} confirms the conjecture that the revival effect is prevalent in a large class of quantum systems with discrete spectra, when the eigenpairs asymptotic supports it  \cite[Section 6.2, page 116]{Brandt_2012}, irrespective of whether the underlying operator is self-adjoint. On the other hand, the present approach might provide a rigorous foundation for tackling the general conjecture formulated in \cite[page 12-13]{chen2013dispersion}. The latter states that a linear PDE with  a dispersion relation that is asymptotic to a polynomial with integer coefficients, in the large wave-numbers regime, should support a type of revival. Further numerical evidence strengthening the validity of this conjecture can be found in \cite{olver2018revivals} for the case $V=0$ and various classes of boundary conditions.  

\medskip

The structure of the paper is as follows. In Section~\ref{section2} we lay down the precise eigenpairs asymptotics, in terms of $V$, that allow the validity of Theorem~\ref{theorem2}. All the results that we present in that section are classical, but we include crucial details of their proofs. Section~\ref{section3}  is devoted to the proof of Theorem~\ref{theorem2}. In the final Section~\ref{section4}, we illustrate our main results by means of examples involving complex Mathieu potentials and discuss the optimality of the different assumptions on $V$.  

\section{Spectral asymptotics} \label{section2}

Let $V\in H^2(0,\pi)$ be a complex-valued potential function. Denote the Hamiltonian associated to \eqref{pbm1} by
\[
L=-\partial_x^2+V:H^2(0,\pi)\cap H^1_0(0,\pi)\longrightarrow L^2(0,\pi).
\] 
Since \[\|V\|_{\infty}=\max_{x\in[0,\pi]}|V(x)|<\infty,\] then the operator $L$ is closed in the domain above, it has a compact resolvent and its adjoint $L^*=-\partial_x^2+\overline{V}$ has the same domain. In some of the statements below we impose the extra condition $\|V\|_{\infty}<\frac32$ of the theorem. Moreover, without loss of generality, we assume in what follows that $\langle V\rangle=0$. 

The boundary value problem  \eqref{pbm1} can be written concisely as  
\begin{align} \label{sv_bc}
    &u_t=-iLu,  \\
   & u(\cdot,0)=f,\nonumber
\end{align}
for an initial value $f\in L^2(0,\pi)$. We know from the classical theory of perturbations of one-parameter semigroups that the operator $iL$ is the generator of a $C_0$ one-parameter semigroup so the equation has a unique solution in $L^2$ for any $f\in L^{2}$. In general, the spectrum of $L$ is not real. However, it is asymptotically close to the real line and the eigenfunctions are asymptotically close to trigonometric functions.  
Our objective in this section is to determine the leading order of these asymptotics and the precise decay rate of the corrections, for the appearence of weak revivals in \eqref{sv_bc}. We then give the proof of Theorem \ref{theorem2} in the next section. 

The next two lemmas are routine consequences of classical properties of non-self-adjoint Sturm-Liouville operators and their analytic perturbation theory, but we give full details of their validity as they are not standard. They imply that the operator $L$ has an infinite sequence of eigenpairs (eigenfunctions and eigenvalues) 
\[\{y_j,w_j^2\}_{j=1}^\infty\subset (H^2\cap H^1_0)\times \mathbb{C}\]
with an asymptotic structure close enough to that of the case $V=0$.

\begin{lemma} \label{eigenvalues_of_L}
Let $w_j^2$ be the eigenvalues of $L$. Then, for $j\to\infty$,
\[
       |w_j-j|=\frac{a_3}{j^3}+O(j^{-4})
\] 
where $a_3\in \mathbb{C}$ is a constant that only depends on $V$. Moreover, if $\|V\|_{\infty}<\frac32$, then each eigenvalue $w_j^2$ is simple.
\end{lemma}
\begin{proof}[Proof]
By virtue of the classical Marchenko asymptotic formula \cite[Theorem~1.5.1]{Marchenko1986}, we know that
\[
     w_j=j+\frac{a_1}{2j}+\frac{\tilde{a}_3}{8j^3}+O(j^{-4}) \qquad k\to\infty
\] 
where $a_1=\langle V\rangle=0$. This gives the eigenvalue asymptotics. 

The family of operators $\alpha\longmapsto T_\alpha=-\partial_x^2+\alpha V$ on the domain $H^2\cap H^1_0$ is a holomorphic family of type (A) for $\alpha\in\mathbb{C}$; see \cite[Example~2.17, p.385]{Kato1980}. As the operator $T_0$ has a compact resolvent, then it follows that $T_\alpha$ have compact resolvent for all $\alpha\in\mathbb{C}$. 

Now, assume that $\|V\|_{\infty}<\frac32$. Then, for $|\alpha|\leq 1$, all the eigenvalues of the family $T_\alpha$ are simple, as they lie in the $\alpha\|V\|_{\infty}$-neighbourhood of $\{j^2\}_{j=1}^\infty$. In particular, for $\alpha=1$ all the eigenvalues of $L$ are indeed simple.
\end{proof}

We now show that the eigenfunctions of $L$ are close enough to the orthonormal Fourier-sine basis corresponding to $V=0$,
\begin{equation} \label{onb}
    d_j(x)=\sqrt{\frac{2}{\pi}}\sin(j\pi x), \quad j\in\mathbb{N}.
\end{equation}
From the second statement of the next lemma, it follows that the solution to \eqref{sv_bc} is given by
\[
   u(x,t)=\sum_{j=1}^\infty \langle f,y_j^*\rangle e^{-iw_{j}^{2} t} y_j(x)
\]
for $y_j^*$ the eigenfunctions of $L^*$ scaled to form a bi-orthogonal set paired with $y_j$, $\langle y_j,y_k^*\rangle =\delta_{jk}$, and the series converges in $L^2$. This turns out to be crucial for the validity of Theorem~\ref{theorem2}.

\begin{lemma} \label{claim_eigenfunctions_L}
In the asymptotic regime $j\to\infty$, the eigenfunctions of $L$ are such that
\begin{equation} \label{asymp_efu}
y_j(x)=c\left[\sin(jx)-\frac{\cos(jx)V_1(x)}{2j}+R_j(x)\right]  
\end{equation}
where $\|R_j\|_{\infty}=O(j^{-2})$ and $V_1$ is defined by
\be\label{v1}
    V_1(x)=\int_0^x V(s)\,\mathrm{d}s.
\ee 
If $\|V\|_{\infty}<\frac32,$ then
\[
     \left\{ \frac{y_j}{\|y_j\|_2}\right\}_{j=1}^\infty  
\]
is a Riesz basis for $L^2(0,\pi)$.
\end{lemma}

\begin{proof}
According to \cite[Lemma~1.4.1]{Marchenko1986}, the eigenfunction associated to $w_k$ is
\[
    y_j(x)=A^+\Phi_j^+(x)+A^-\Phi_j^-(x)
\]
where
\[
    \Phi_{j}^\pm(x)=e^{\pm iw_j x}\left(1\pm \frac{V_1(x)}{2i w_j}+\frac{V_2(x)}{(2iw_j)^2}+R_j^{\pm}(x)\right)
\]
for $\|R^{\pm}_j\|_{\infty} =O(w_j^{-3})$, $V_1$ given by \eqref{v1} and 
\begin{gather*}
\begin{aligned} V_2(x)&=\int_{0}^x LV_1(s)\,\mathrm{d}s \\ &= \int_0^xV(s)V_1(s)\,\mathrm{d}s-V(x) + V(0). \end{aligned}
\end{gather*}
Moreover, again from \cite[Lemma~1.4.1]{Marchenko1986}, we know that $R^{\pm}_{j}(0) = 0$. Thus, substituting the boundary conditions $y_j(0)=0$, we get $-A^{-} = A^{+}$. Set the latter equal to 1.
Hence,
\begin{gather*}
     y_j(x)=\sin(w_jx)-\frac{\cos(w_j x)V_1(x)}{2w_j}+\tilde{R}_j(x) \quad \text{where} \\ \quad \|\tilde{R}_j\|_{\infty}=O(j^{-2}).
\end{gather*}
Now,
\begin{align*}
    \sin\big(jx+O(j^{-3})x\big)&=\sin(jx)+s_j(x) \qquad \text{and}\\
   \cos\big(jx+O(j^{-3})x\big)&=\cos(jx)+c_j(x)
\end{align*}
where
\begin{align*}
    |s_j(x)|+|c_j(x)| &\leq 2\big|\cos(O(j^{-3})x)-1\big|+2\big|\sin(O(j^{-3})x)\big| \\ & \leq k_1j^{-6}+k_2j^{-3} 
\end{align*}
for all $x\in[0,\pi]$. This gives \eqref{asymp_efu}, by taking $R_j(x)=\tilde{R}_j(x)+s_j(x)+c_j(x)$. 

Let us now show that if $\|V\|_{\infty}<\frac32$, then the eigenfunctions form a Riesz basis. We aim at applying \cite[Theorem~2.20, p.265]{Kato1980}. According to \cite[Theorem~1.3.1]{Marchenko1986} combined with Lemma~\ref{eigenvalues_of_L}, the family of eigenfunctions $\{y_j\}$ is complete in $L^2(0,\pi)$. Since it has a dual pair, $\{y^*_j\}$, then it is minimal and so therefore exact \cite{heil2011}. Minimality ensures that $\{y_j\}$ is $\omega$-independent \cite{heil2011}. This gives two of the hypotheses of \cite[Theorem~2.20, p.265]{Kato1980}.

Now, by virtue of \eqref{asymp_efu} already proven, there exists a constant $c_4>0$ such that 
\[
\left\|\frac{y_j}{\|y_j\|_2}-d_j\right\|_2\leq \frac{k_3}{j}.\]
Thus,
\[
     \sum_{j=1}^\infty \left\|\frac{y_j}{\|y_j\|_2}-d_j\right\|_2^2<\infty.
\]
This is the other hypothesis required in \cite[Theorem~2.20, p.265]{Kato1980} and so indeed\footnote{The conclusion of \cite[Theorem~2.20, p.265]{Kato1980} does not exactly state that the family is a Riesz basis. But the proof  implies that the family is equivalent to an orthonormal basis, hence it is indeed always a Riesz basis.} $\left\{\frac{y_j}{\|y_j\|}\right\}_{j=1}^\infty$ is a Riesz basis of $L^2(0,\pi)$.
\end{proof}

We discuss the optimality of the condition $\|V\|_\infty<\frac32$ and the case where $V$ is real-valued in Section~\ref{optimality}.

From the above, we gather that the eigenvalues of $L$ are
\begin{equation} \label{true_hyp_eva}
        \lambda_j=w_j^2=j^2+\frac{k_j}{j^2} \qquad \text{for suitable }\{k_j\}\in \ell^{\infty}.
\end{equation}
Let
\[
     n_j(x)=\sqrt{\frac{2}{\pi}}\cos(j\pi x), \quad j\in\mathbb{N}
\]
be the non-constant orthonormal Fourier-cosine basis. 
Below we fix the eigenfunctions according to a normalisation of their bi-orthogonal pairs. Concretely, let   
\begin{equation} \label{true_hyp_eve}
      \phi_j(x)=\gamma_j y_j(x)=\gamma_jd_j(x)-\frac{\gamma_j n_j(x)V_1(x)}{2j}+\gamma_j R_j(x) 
\end{equation}
where $\|R_j\|_\infty\leq \frac{c}{j^2}$. Without further mention, from now on the non-zero constants $\gamma_j$ are chosen, such that the associated bi-orthogonal sequence $\{\phi^*_j\}$ is given by 
\begin{equation} \label{forelstar}
\phi_j^*(x)=d_j(x)-\frac{n_j(x)\overline{V}_1(x)}{2j}+R_j^*(x) .
\end{equation}
Then, there exist constants $0<\tilde{\gamma}_1<\tilde{\gamma}_2<\infty$ such that
\begin{equation}   \label{asympt_gammas}
    \frac{\tilde{\gamma}_1}{j}<|\gamma_j-1|<\frac{\tilde{\gamma}_2}{j} \qquad \text{and} \qquad
      \frac{\tilde{\gamma}_1}{j}<\big|\|\phi^*_j\|_2-1\big|<\frac{\tilde{\gamma}_2}{j},
\end{equation}
for all $j\in\mathbb{N}$.

\section{Proof of Theorem~\ref{theorem2}}
\label{section3}
We now state and prove a crucial lemma, from which Theorem~\ref{theorem2} follows as a corollary.
 
\begin{lemma} \label{claim3}
If the potential $V$ is such that $\langle V\rangle=0$ and $\|V\|_{\infty}<\frac32$, then the solution to the time-evolution equation \eqref{sv_bc} is given by
\begin{equation}  \label{weak_revival_sc}
 u(x,t)=w(x,t)+\sum_{j=1}^\infty \langle f,d_j\rangle e^{-ij^2t}d_j(x), 
\end{equation}
where for each fixed $t>0$, $w(\cdot,t)\in C([0,\pi])$.  
\end{lemma}
\begin{proof}
We separate the proof into four steps.

\underline{Step 1}. Consider the $L^2$ expansion of the initial condition, 
\[
    f(x)=\sum_{j=1}^\infty c_j \phi_j(x), \quad c_{j} = \langle f, \phi^{*}_j\rangle.
\]
According to \eqref{true_hyp_eva},
\begin{align*}
     u(x,t)&=\sum_{j=1}^\infty c_j e^{-i\lambda_j t}\phi_j(x) 
     =\sum_{j=1}^\infty c_j e^{-i\left( j^2+\frac{k_j}{j^2} \right) t}\phi_j(x) \\
&=  \sum_{j=1}^\infty c_j e^{-i j^2t} \left( 1-\frac{i k_j}{j^2}\int_0^t e^{-\frac{ik_j}{j^2}s}\,\mathrm{d}s  \right) \phi_j(x) \\
&= U_1(x,t)-U_2(x,t)
\end{align*}
where
\[
      U_1(x,t)= \sum_{j=1}^\infty c_j e^{-ij^2t}\phi_j(x)
\]
and $U_2(x,t)$ has a similar expression but involving the integral above. We treat these two terms separately.

\underline{Step 2}. Let us show that $U_2\in C^1([0,\pi])$. Set
\[
      \zeta_j(x,t)=\frac{ic_j k_j}{j^2} e^{-ij^{2}t}\int_0^t e^{-\frac{ik_j}{j^2}s}\,\mathrm{d}s \ \phi_j(x).  
\]
Then,
\[
    |\zeta_j(x,t)|\leq \frac{\|\{k_j\}\|_{\infty}  \|\phi_j\|_{\infty}|\langle v,\phi_j^*\rangle|}{j^2} tB_j\leq \frac{\|\{k_j\}\|_{\infty} \|\phi_j\|_{\infty}\|v\|_2\|\phi_j^*\|_2}{j^2} tB_j
\]
where
\[
    B_j=\sup_{s\in [0,t]}\left| e^{-\frac{ik_j s}{j^2}} \right|
   \leq \sup_{s\in [0,t]} e^{\left|\frac{\operatorname{Im} k_j s}{j^2}\right|}\leq B<\infty.
\]
According to \eqref{true_hyp_eva}, the right hand side constant is independent of $j$. Here $t$ is fixed. Moreover, by virtue of \eqref{true_hyp_eve} and \eqref{asympt_gammas},
\[
    \max\big\{\|\phi^*_j\|_2,\|\phi_j\|_{\infty}\big\}\leq c
\]
for all $j\in \mathbb{N}$, where the constant $c>0$ is independent of $j$.
Hence, by Weierstrass' M-test, 
\[
    U_2(x,t)=\sum_{j=1}^\infty \zeta_j(x,t)
\]
converges absolutely and uniformly to a $C^1$ function, because each component is $C^1$. 

\underline{Step 3}. Consider now $U_1$. According to \eqref{true_hyp_eve},
\begin{align*}
    U_1(x,t) &= \sum_{j=1}^\infty c_j\gamma_j e^{-ij^2 t}d_j(x)-\sum_{j=1}^\infty \frac{c_j\gamma_j e^{-ij^2 t}}{2j}n_j(x)V_1(x) +\sum_{j=1}^\infty c_j\gamma_j e^{-ij^2 t} R_j(x) \\
         &= u_3(x,t)+u_4(x,t)+u_5(x,t).
\end{align*} 
In this step we show that $u_4$ and $u_5$ are continuous in the variable $x$.
From the asymptotic behaviour of $R_j$ and an identical argument as we used in step~2, we know that $u_5$ is $C^1$ in the variable $x$. 

Now, for $u_4(x,t)$, note that 
\[
    \sum_{j=1}^\infty |c_j|^2<\infty,
\]
because $\{\phi_j^*\}$ is a Riesz basis, e.g. \cite[Theorem~7.13]{heil2011}. Then,
by Cauchy-Schwarz, \[
     \sum_{j=1}^\infty \left|\frac{c_j}{j}\right|<\infty.
\]
Thus, since $V_1$ is a bounded function, for all fixed $t$, the sequence
\[
     \left\{\frac{c_j\gamma_j e^{-ij^2 t}\|V_1\|_{\infty}}{2j}\right\}_{j=1}^\infty\in \ell^1(\mathbb{N}).
\]
Hence, for all fixed  $t>0$, the family of sequences (family for $x\in[0,\pi]$) 
\[
      \left\{\frac{c_j\gamma_j e^{-ij^2 t}n_j(x)V_1(x)}{2j}\right\}_{j=1}^\infty\in \ell^1(\mathbb{N}).
\]
Therefore, by the Dominated Convergence Theorem (in $\ell^1$), we have that
\[
    \lim_{x\to x_0} u_4(x,t)=u_4(x_0,t),
\]
for all $x_0\in[0,\pi]$. That is, $u_4$ is a continuous function of the variable $x$. 

\underline{Step~4}. Finally, we consider $u_3(x,t)$. 
By \eqref{forelstar}, we know that
 \[
     c_j=\langle f,d_j\rangle -\frac{\langle f V_1,n_j\rangle}{2j}+\langle f,R_j^*\rangle.
\]
Then,
$
    u_3(x,t)=u_6(x,t)-u_7(x,t)+u_8(x,t),
$
where
\begin{gather*}
       u_6(x,t)=\sum_{j=1}^\infty  \langle f,d_j\rangle \gamma_{j} e^{-ij^2 t}d_j(x),
\qquad u_7(x,t)=\sum_{j=1}^\infty \frac{\langle f V_1,n_j\rangle}{2j} \gamma_j e^{-ij^2 t}d_j(x) \\ \text{and} \quad      u_8(x,t)=\sum_{j=1}^\infty \langle f,R_j^*\rangle \gamma_j e^{-ij^2 t}d_j(x).
\end{gather*}

We write $\gamma_{j} = 1+(\gamma_{j} - 1)$ and split each term of $u_{6}(x,t)$, $u_{7}(x,t)$, $u_{8}(x,t)$ into two sums. 

For $u_{6}(x,t)$ we have
\[
u_{6}(x,t) =\sum_{j=1}^\infty  \langle f,d_j\rangle e^{-ij^2 t}d_j(x) + \sum_{j=1}^\infty  \langle f,d_j\rangle (\gamma_{j} -1) e^{-ij^2 t}d_j(x)
\]
The first component of $u_{6}$ is the second component in the solution representation \eqref{weak_revival_sc}. To deal with the second one we use \eqref{asympt_gammas}. By Cauchy-Schwarz,
\[
\sum_{j=1}^{\infty} \frac{|\langle f, d_{j}\rangle|}{j} < \infty.
\]
Hence, by Weierstrass' M-test again, for each $t>0$, the series of functions
\[
\sum_{j=1}^\infty  \langle f,d_j\rangle (\gamma_{j} -1) e^{-ij^2 t}d_j(x)
\]
converges absolutely and uniformly to a $C^{1}$ function on $[0,\pi]$.
 
Now, the function $u_7$ is written as follows 

\[
u_{7}(x,t) = \sum_{j=1}^\infty \frac{\langle f V_1,n_j\rangle}{2j} e^{-ij^2 t}d_j(x) + \sum_{j=1}^\infty \frac{\langle f V_1,n_j\rangle}{2j} (\gamma_j - 1) e^{-ij^2 t}d_j(x).
\]

The first component of $u_{7}$ is continuous as a consequence of an argument similar to that employed for $u_4$. Indeed, since $\{n_j\}$ is an orthonormal basis and $f V_1\in L^2(0,\pi)$,
\[
   \sum_{j=1}^\infty \left|\frac{\langle f V_1,n_j\rangle e^{-ij^2 t}}{2j}\right|\leq \| f V_1\|_2\frac{\pi}{\sqrt{6}},
\]
so we can use the sequence 
\[
    \left\{\frac{\langle f V_1,n_j\rangle e^{ij^2 t}}{2j}\right\}_{j=1}^\infty\in \ell^1
\]
to ensure continuity, via the Dominated Convergence Theorem. The second component is a $C^{1}$ function since its Fourier-sine coefficients decay like $j^{-2}$, due to \eqref{asympt_gammas} and the Cauchy-Schwarz inequality. So, in total, $u_{7}(\cdot,t)\in C([0,\pi])$. 

Finally, for the function $u_8$ we have that
\[
u_8(x,t) = \sum_{j=1}^\infty \langle f,R_j^*\rangle e^{-ij^2 t}d_j(x) + \sum_{j=1}^\infty \langle f,R_j^*\rangle (\gamma_j-1) e^{-ij^2 t}d_j(x).
\]
Here, the first component is $C^1$ in $x$. Indeed,
since
\[
     |\langle f,R_j^*\rangle|=\left|\int_0^\pi f(x)\overline{R_j^*(x)}\,\mathrm{d}x\right| \leq \frac{\pi \|f\|_2 c}{j^2},
\]
the first component is a function whose sine-Fourier coefficients decay like $j^{-2}$, so it is continuously differentiable in $x$. The second component is twice differentiable in $x$, since due to \eqref{asympt_gammas} it represents a function whose Fourier-sine coefficients decay like $j^{-3}$. Thus, $u_{8}(\cdot,t)$ belongs to $C^{1}([0,\pi])$. 

Collecting the statements about $u_k$ from the previous steps, we conclude that the expression of $u$ is as claimed in \eqref{weak_revival_sc} where indeed $w$ is continuous in the variable $x$.
\end{proof}

\begin{remark}
In the proof of this lemma, note that all components of $w$ are $C^1$, except $u_4$ and $u_7$.  
\end{remark}

The proof of Theorem~\ref{theorem2} now follows from Lemma~\ref{claim3} and the combinatorial argument for $V=0$.

\begin{proof}[Proof of Theorem~\ref{theorem2}]
We show that, if $V=0$, then the solution to \eqref{sv_bc} at rational times $t_{\mathrm{r}}= 2\pi\frac{p}{q}$ is 
\begin{equation}\label{classicalrevivals}
     u(x,t_{\mathrm{r}})=\sum_{j=1}^\infty \langle f,d_j\rangle e^{-ij^2t_{\mathrm{r}}}d_j(x) =  
      \frac{1}{q}\sum_{k,m=0}^{q-1} e^{2\pi i (-m^2 \frac{p}{q} +m\frac{k}{q})} f^{\mathrm{o}}\Big(x-2\pi\frac{k}{q}\Big),
\end{equation}
where $f^{\mathrm{o}}$ denotes the odd, $2\pi$-periodic extension of $f$. Therefore, replacing $V$ with $V-\langle V \rangle$ if needed, and applying Lemma~\ref{claim3}, gives Theorem~\ref{theorem2}.

The proof of \eqref{classicalrevivals} is as follows. For $t\in\mathbb{R}$ we have that
\[
     u(x,t)=\frac{1}{2\pi} \sum_{j=-\infty}^\infty e^{-ij^2t} \langle f^{\mathrm{o}},e^{ij(\cdot)}\rangle e^{ijx}.
\] 
Then, for $t=t_{\mathrm{r}}$ take $j\underset{q}{\equiv} m$ so that $e^{ij^2t_{\mathrm{r}}}=e^{im^2t_{\mathrm{r}}}$. Thus,
\[
      u(x,t_{\mathrm{r}})=\frac{1}{2\pi} \sum_{m=0}^{q-1} e^{-im^2t_{\mathrm{r}}} \sum_{\substack{j\in \mathbb{Z} \\ j\underset{q}{\equiv} m}}\langle f^{\mathrm{o}},e^{ij(\cdot)}\rangle e^{ijx}.
\]
Let the summation on the right had side be denoted by $T$. Since
\[
   \sum_{k=0}^{q-1}e^{2\pi i(m-j)\frac{k}{q}}=\begin{cases} q & j\underset{q}{\equiv} m \\ 0 & j\underset{q}{\not\equiv} m,\end{cases} 
\]
we have 
\[ \begin{aligned}T&=\frac{1}{q}\sum_{k=0}^{q-1} e^{2\pi im\frac{k}{q}}\sum_{j\in\mathbb{Z}} e^{-2\pi i\frac{k}{q}j}\langle f^{\mathrm{o}},e^{ij(\cdot)}\rangle e^{ijx} \\ &  = \frac{1}{q}\sum_{k=0}^{q-1} e^{2\pi im\frac{k}{q}}\sum_{j\in\mathbb{Z}} \Big\langle f^{\mathrm{o}}\Big(\cdot-\frac{2\pi k}{q}\Big),e^{ij(\cdot)}\Big\rangle e^{ijx} \\ &= \frac{1}{q}\sum_{k=0}^{q-1} e^{2\pi im\frac{k}{q}}f^{\mathrm{o}}\Big(x-\frac{2\pi k}{q}\Big) .\end{aligned}\]
Hence, \eqref{classicalrevivals} holds true. 
\end{proof}

\begin{remark} \label{Duhamel}
For general bounded complex potential $V$, we know from the Dyson expansion that \eqref{weak_revival_sc} holds true for
\[
    w(x,t)=\sum_{k=1}^\infty w_k(x,t),
\] 
where $w_k(x,t)$ are explicitly given in terms of integrals of regular functions. These functions are continuous for $x\in (0,\pi)$. By tracking the convergence of the series, it might be possible to establish its continuity, therefore extend the results of this section to all $\|V\|_{\infty}<\infty$. See the ideas described in \cite[\S7]{Dimoudis22}.
\end{remark}


\section{The hypotheses of Theorem~\ref{theorem2}} \label{optimality}
\label{section4}
In this section we examine the optimality of the hypotheses of Theorem~\ref{theorem2}, in terms of the size and the regularity of the potential. 

Firstly, we note that for $L$ self-adjoint, the assumptions on $V$ can be relaxed. According to the asymptotic expansions reported in \cite[eqs. $(4.21)_{7}$, $(5.4)_{2}$ and $(5.9)_{2}$]{FultonPruess1994}, the identity \eqref{asymp_efu} is valid for any $V:[0,\pi]\longrightarrow \mathbb{R}$ of bounded variation irrespective of the size of $\|V\|_{\infty}<\infty$. Therefore, by following the same method of proof presented above, it directly follows that the conclusion of Theorem~\ref{theorem2} still holds true under these modified hypotheses on $V$. That is, the solution to \eqref{pbm1} at rational times, is given by
\begin{equation} \label{repeatedrevival} 
     u\Big(x,2\pi\frac{p}{q}\Big)=w\Big(x,2\pi\frac{p}{q}\Big)+\frac{1}{q} \ e^{-2\pi i\langle V\rangle \frac{p}{q}}
      \sum_{k,m=0}^{q-1} e^{ 2\pi i ( m\frac{k}{q}-m^2 \frac{p}{q})} f^{\mathrm{o}}\Big(x-2\pi\frac{k}{q}\Big)
\end{equation}
for a suitable function $w(\cdot,t)$ continuous in $x\in[0,\pi]$. 

In the more general non-self-adjoint setting, for $\|V\|_{\infty}>\frac{3}{2}$, we only know from the available asymptotic formulas that, for large wavenumbers, all the eigenvalues are simple and the corresponding eigenfunctions form a basis of a subspace $\mathcal{S}$ of finite co-dimension. Despite of this, still $iL$ is the generator of a one-parameter semigroup, see Remark~\ref{Duhamel} above. The solution to \eqref{pbm1} exists and it is unique for all $f\in L^2$. Moreover, \eqref{sv_bc} has a solution with an $L^2$-convergent eigenfunction expansion and a version of Theorem~\ref{theorem2} can be recovered for all $f\in \mathcal{S}$.

\medskip

We now present an example of a purely imaginary $V\in C^{\infty}$ with $\|V\|_\infty=2$, for which \eqref{repeatedrevival} appears to still be valid. For this purpose we choose for $V$ a purely imaginary Mathieu potential.  

Let $q\in\mathbb{C}$ and $V(x)=2q\cos(2x)$. Then $\langle V \rangle=0$. The eigenvalue equation associated to the operator $L$ is Mathieu's equation. The eigenvalues of $L$ are 
$
\omega_j^2=b_j(q),
$
the Mathieu characteristic values,
which satisfy
\[
b_j(q)=j^2+\frac{1}{2(j^2-1)}q^2+O(q^4)
\]
as $|q|\to 0$.
The corresponding eigenfunctions are the Mathieu functions
\[
    \phi_j(x)=\operatorname{se}_j(x,q)
\]  
for $j\in \mathbb{N}$. See \cite[\S 7.4]{Ince1956} and also \cite[28.2-7]{Olver10}. 

In Figure~\ref{fig1} we set $f(x)=\chi_{[\frac{3\pi}{8},\frac{5\pi}{8}]}(x)$ and show a numerical approximation to 100 modes of $u(x,t)$ at time $t=\frac{2\pi}{5}$ for purely imaginary $q$ with increasing modulus. As $|q|$ increases, we show how the correction $w(x,t)$ affects the revivals part of the solution. Note that the conclusions of Theorem~\ref{theorem2} only hold for $|q|<\frac{3}{4}$, but the numerical approximation in the figure suggests that this conclusion appears to be valid also when $\|V\|_{\infty}=2$ for this potential. 

The graphs shown in Figure~\ref{fig2} re-inforce the conjecture that the correction term $w$ is continuous beyond the threshold $\|V\|_\infty=\frac32$.
Indeed, in Figure~\ref{fig2} we show a 100 modes approximation of 
\[
u(x,t)-\sum_{j=1}^\infty \langle f,d_j\rangle e^{-ij^2t}d_j(x) 
\]
for the same data as in Figure~\ref{fig1}. For (a)-(b) we confirm the shape of $w(x,t)$. For (c)-(d), note that even when $q=\frac{3\pi}{4}i$ and $q=i$, the difference appears to still be continuous.

\section*{Acknowledgements}
We kindly thank David Smith for his valuable comments made during the preparation of this manuscript. The work of GF was funded by EPSRC through Heriot-Watt University support for Research Associate positions, under the Additional Funding Programme for the Mathematical Sciences. BP was partially supported by a Leverhulme Research Fellowship.

\bibliographystyle{amsplain}

\begin{figure}
(a) \includegraphics[width=.4\textwidth]{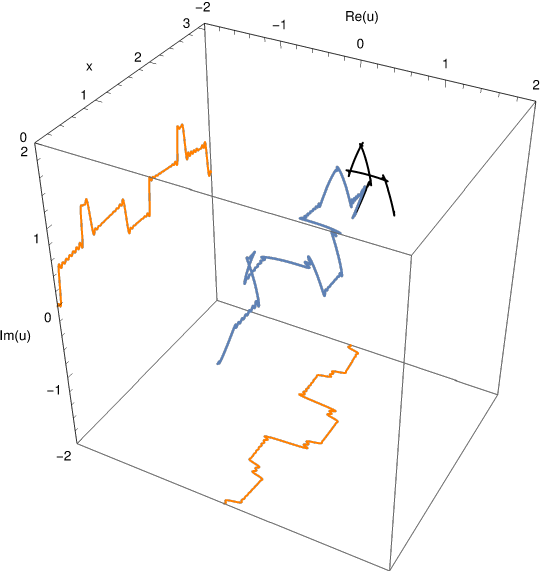} \ 
(b)\includegraphics[width=.4\textwidth]{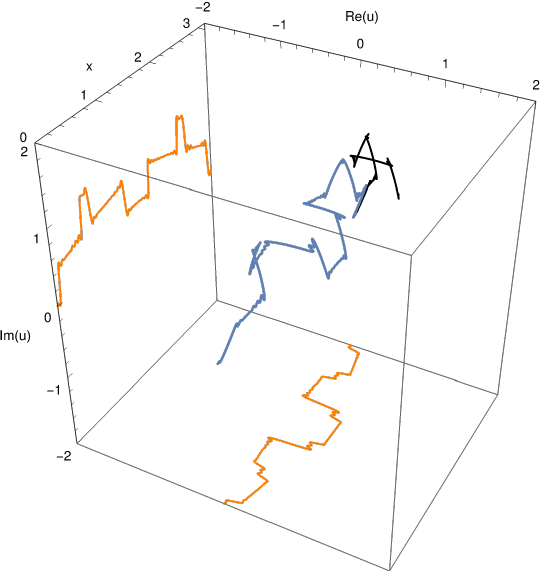}

(c) \includegraphics[width=.4\textwidth]{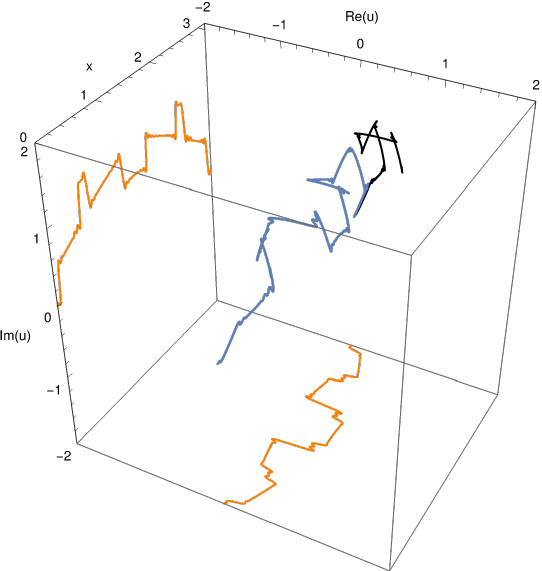} \ 
(d)\includegraphics[width=.4\textwidth]{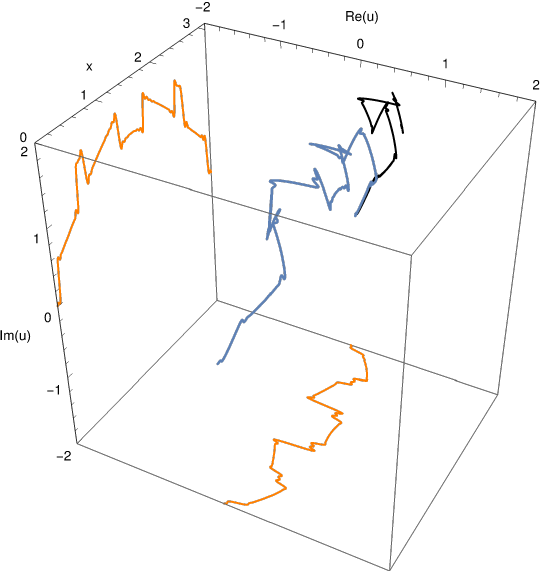}
\caption{Here $V(x)=2q\cos(2x)$ and $t=\frac{2\pi}{5}$. We show an approximation of $u(x,t)$ to 100 modes. The blue curves are the solutions as complex-valued functions of $x$, the orange curves correspond to projections of the real and imaginary parts of these solutions, and the black curves are the projections corresponding to the curves traced by the solutions in the complex plane for $x\in[0,\pi]$. The figures shown match (a)~$q=\frac{i}{4}$, (b)~$q=\frac{i}{2}$, (c)~$q=\frac{3i}{4}$ and (d)~$q=i$. \label{fig1}}
\end{figure}

\begin{figure}
(a) \includegraphics[width=.4\textwidth]{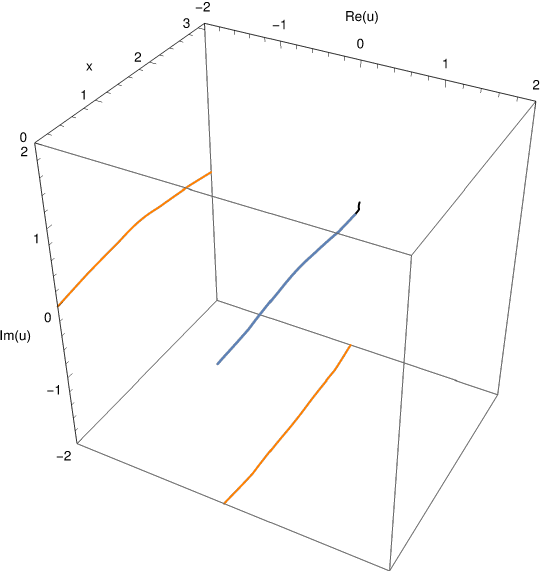} \ 
(b)\includegraphics[width=.4\textwidth]{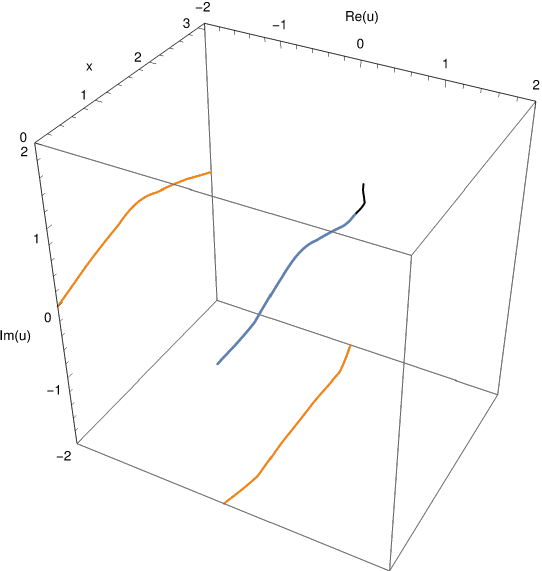}

(c) \includegraphics[width=.4\textwidth]{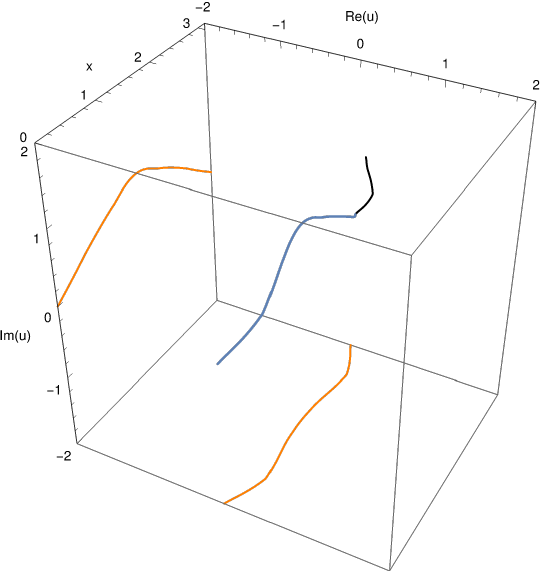} \ 
(d)\includegraphics[width=.4\textwidth]{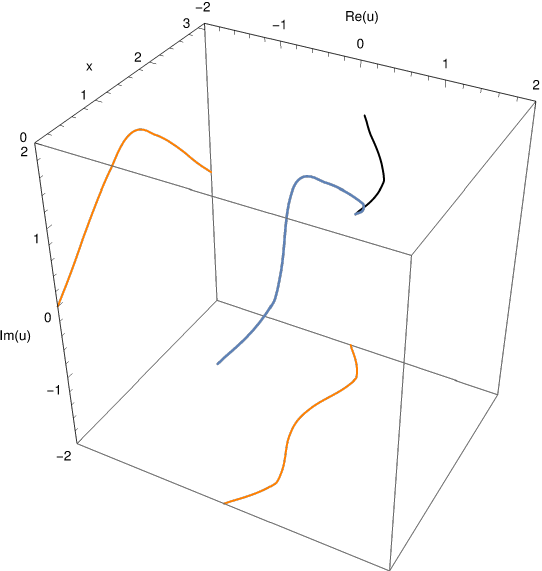}
\caption{Here $V(x)=2q\cos(2x)$ and $t=\frac{2\pi}{5}$. We show an approximation of $w(x,t)$ to 100 modes. The blue curves are the complex-valued functions of $x$, the orange curves correspond to projections of the real and imaginary parts of these, and the black curves correspond to the graphs traced by $w(x,\frac{2\pi}{5})$ on the complex plane for $x\in[0,\pi]$. The figures shown match  (a)~$q=\frac{i}{4}$, (b)~$q=\frac{i}{2}$, (c)~$q=\frac{3i}{4}$ and (d)~$q=i$. \label{fig2}}
\end{figure}

\end{document}